\documentclass{amsart}

\usepackage[T1]{fontenc}
\usepackage[latin1]{inputenc}
\usepackage{amsmath}
\usepackage{amssymb}
\usepackage{amsthm}
\usepackage{dsfont}
\usepackage{color}

\theoremstyle{plain}\newtheorem{theo}{Theorem}
\theoremstyle{plain}
\theoremstyle{definition}
\theoremstyle{plain}\newtheorem{defi}{Definition}[section]
\theoremstyle{plain}\newtheorem{lem}[defi]{Lemma}
\theoremstyle{plain}\newtheorem{prop}[defi]{Proposition}
\theoremstyle{definition}

\newcommand{\N}{{\mathds N}}
\newcommand{\R}{{\mathds R}}
\newcommand{\Z}{{\mathds Z}}
\DeclareMathOperator{\var}{Var}

\DeclareMathOperator{\sgn}{sgn}

\begin{document}

\title[Empirical Process of Random Walk in Random Scenery]{The Sequential Empirical Process of a Random Walk in Random Scenery}


\author[M. Wendler]{Martin Wendler}
\address{Ernst-Moritz-Arndt-Universit\"at Greifswald, 17487 Greifswald, Germany}
\email{martin.wendler@uni-greifswald.de}
\date{\today}


\begin{abstract} A random walk in random scenery $(Y_n)_{n\in\N}$ is given by $Y_n=\xi_{S_n}$ for a random walk $(S_n)_{n\in\N}$ and iid random variables $(\xi_n)_{n\in\Z}$. In this paper, we will show the weak convergence of the sequential empirical process, i.e. the centered and rescaled empirical distribution function. The limit process shows a new type of behavior, combining properties of the limit in the independent case (roughness of the paths) and in the long range dependent case (self-similarity).
\end{abstract}


\maketitle

\section{Introduction} For a stationary, real valued sequence $(Y_n)_{n\in\N}$ of random variables with marginal distribution function $F$, the empirical distribution function $F_n$ is defined by
\begin{equation}
F_n(t)=\frac{1}{n}\sum_{i=1}^n\mathds{1}_{\{Y_i\leq t\}}.
\end{equation}
If the marginal distribution function $F$ is continuous, we can without loss of generality assume that $F(t)=t$ (otherwise replacing $Y_n$ by $F(Y_n)$). The sequential empirical process is a two-parameter stochastic process $\big(W_n(s,t)\big)_{s,t\in[0,1]}$ defined by
\begin{equation}
W_n(s,t)=\sum_{i=1}^{[ns]}\left(\mathds{1}_{\{Y_i\leq t\}}-t\right),
\end{equation}
where $[x]$ denotes the integer part of $x$. Note that we will have to rescale this process in order to obtain weak convergence, but as we need a different scaling for different kinds of stochastic processes, we have not included the scaling here. For iid (independent and identical distributed) random variables $(Y_n)_{n\in\N}$, Donsker \cite{dons} showed the weak convergence of the (non-sequential) empirical process  $\big(\frac{1}{\sqrt{n}}W_n(1,t)\big)_{t\in[0,1]}$ to a Brownian bridge. This was extended by M\"uller \cite{mull} to the sequential empirical process $\big(\frac{1}{\sqrt{n}}W_n(s,t)\big)_{s,t\in[0,1]}$. The limit Gaussian process is the so called Kiefer-M\"uller process $K$, which is self-similar with exponent $b=\frac{1}{2}$, that means for any $a>0$ the process $\big(K(as,t)\big)_{s,t\in[0,1]}$ has the same distribution as $\big(a^{\frac{1}{2}}K(s,t)\big)_{s,t\in[0,1]}$. For fixed $s\in[0,1]$, $\left(K(s,t)\right)_{t\in[0,1]}$ is a Brownian bridge, while for fixed  $t\in[0,1]$ $\left(K(s,t)\right)_{s\in[0,1]}$ is a Brownian motion. This implies that there is an almost surely continuous modification of $K$, but the paths are not $\gamma$-H\"older continuous for any $\gamma>\frac{1}{2}$.

This limit theorem has been extended to different kinds of short range dependent processes $(Y_n)_{n\in\N}$, where one still needs a $n^{-\frac{1}{2}}$ scaling and the limit process is still self-similar with exponent $\frac{1}{2}$. For example, Berkes and Philipp \cite{berk} studied approximating functionals of strongly mixing sequences and Berkes, H\"ormann, Schauer \cite{ber2} so called $S$-mixing random variables. In the short range dependent case, the limit process is for fixed $t\in[0,1]$ a Brownian motion as in the independent case, so the paths are not smoother.

For long range dependent processes, the limit behavior is different in many aspects. For Gaussian sequences with slowly decaying covariances, Dehling and Taqqu \cite{dehl} showed the convergence of sequential empirical process to a limit process that is self-similar with exponent $b>\frac{1}{2}$ and that is degenerate in the following sense: For fixed $s$, the process is not a Brownian bridge, but a deterministic function multiplied by a random variable. The paths for fixed $s$ might be differentiable. For fixed $t$, the limit process is a fractional Brownian motion which is $\gamma$-H\"older continuous with exponent $\gamma>\frac{1}{2}$. For long range dependent linear processes, analog results were proved by Ho and Hsing \cite{ho}.

In this paper, we will consider the random walk in random scenery, which is often considered to be another model for a long range dependent sequence of random variables. Let $(S_n)_{n\in\N}$ with $S_{n}=\sum_{i=1}^nX_i$ be a random walk in the normal domain of attraction of an $\alpha$-stable L\'evy process (with iid, integer valued increments $(X_n)_{n\in\N}$) and $(\xi_n)_{n\in\Z}$ a sequence of iid random variables (called scenery). Then the stationary process $(Y_n)_{n\in\N}$ with $Y_n=\xi_{S_n}$ is called random walk in random scenery and was first investigated by Kesten and Spitzer \cite{kest} and Borodin \cite{boro}.

The behavior of partial sum process $Z_n$ with $Z_n(s)=\sum_{i=1}^{[ns]} Y_i$ has been studied extensively. It converges weakly to a self-similar process with exponent $b>\frac{1}{2}$, which has smooth paths even if the random variables $(\xi_n)_{n\in\Z}$ are in the domain of attraction of a L\'evy process with jumps, see \cite{kest}. Other results include the law of the iterated logarithm (Khoshnevisan and Lewis \cite{khos}), large deviations (Gantert, K\"onig, and Shi \cite{gant}), extremes (Franke and Saigo \cite{frank}) and $U$-statistics (Guillotin-Plantard and Ladret \cite{guil}, Franke, P\`ene, and Wendler \cite{fran2}). As far as we know, there are no results on the empirical process of a random walk in random scenery.

\section{Main Results}

We will now give a functional non-central limit theorem for the sequential empirical process of a random walk in random scenery, that means the two-parameter process $W_n$ with
\begin{equation}
W_n(s,t)=\sum_{i=1}^{[ns]}\left(\mathds{1}_{\{Y_i\leq t\}}-t\right),\ \ \text{where} \ \ Y_n=\xi_{S_n} \ \  \text{and} \ \ S_n:=\sum_{i=1}^nX_i.
\end{equation}
Let us first introduce the limit process $W$: Let $K=\left(K(x,t)\right)_{x\in\R, t\in[0,1]}$ be a two-sided Kiefer-M\"uller process, which is defined as follows: $K_1=\left(K_1(x,t)\right)_{x\in[0,\infty), t\in[0,1]}$ and $K_{-1}$ be two independent, centered, continuous, two-parameter Gaussian process with covariance
\begin{equation}
E\left[K_i(x,t)K_i(x',t')\right]=\min\{x,x'\}\left(\min\{t,t'\}-tt'\right) \ \ \ \text{for}\ \ \ i=1,-1.
\end{equation}
Set $K(x,t)=K_{\sgn(x)}(|x|,t )$. Furthermore, let $(L_s(x))_{s\geq0}$ be the local time of the limit process $(S_s^\star)_{s\geq0}$ of the rescaled partial sum $(n^{-\frac{1}{\alpha}}\sum_{i=1}^{[ns]}X_i)_{s\geq0}$, that means
\begin{equation}
\int_0^t\mathds{1}_{[a,b)}(S_s^\star)ds=\int_a^bL_t(x)dx.
\end{equation}
For the existence of such a continuous time, see Getoor and Kesten \cite{geto}. Now the limit process $W$ can be described by the following stochastic integral with respect to the (two-sided) Brownian motion $(K(x,t))_{x\in\R}$
\begin{equation}
 W(s,t):=\int_{\R} L_s(x)dK(x,t).
\end{equation}
We will investigate the properties of this process after our main Theorem.
\begin{theo}\label{theo1} Let $(\xi_n)_{n\in\Z}$ be an iid sequence of random variables uniformly distributed on $[0,1]$. If $(X_n)_{n\in\N}$ is another iid sequence, independent of $(\xi_n)_{n\in\Z}$, integer valued and the law of $X_n$ is in the normal domain of attraction of an $\alpha$-stable law $F_\alpha$ with $1<\alpha\leq 2$, then we have the weak convergence
\begin{equation}
n^{-1+\frac{1}{2\alpha}}W_n\Rightarrow W
\end{equation}
in the space $D\left([0,1]^2\right)$.
\end{theo}
The space $D\left([0,1]^2\right)$ is the space of functions from $[0,1]^2$ to $\R$, for which the limit in each quadrant exists and which are continuous in each point coming from the upper right quadrant, equipped with the multidimensional Skorokhod distance (see Bickel and Wichura \cite{bick}). From the definition of $W$, we can see that for fixed $t$, the process $(W(s,t))_{s\in[0,1]}$ is the limit process of the random walk in random scenery as described by Kesten and Spitzer \cite{kest}. It is clear that the process $W$ is self-similar with the same exponent $b=1-\frac{1}{2\alpha}$, that means $(W(as,t))_{s,t\in[0,1]}$ has the same distribution as $(a^{1-\frac{1}{2\alpha}}W(s,t))_{s,t\in[0,1]}$.

On the other hand, for fixed $s$, the process $(W(s,t))_{t\in[0,1]}$ is a mixture of Brownian bridges (or a Brownian bridge with a random variance). So the process $(W(s,t))_{t\in[0,1]}$ has paths with the same properties as a Brownian bridge, and consequently they are $\gamma$-H\"older continuous for all $\gamma<\frac{1}{2}$, but not for any $\gamma>\frac{1}{2}$. In this sense, the limit process combines properties from the independent case (roughness of Kiefer-M\"uller process) and from the long range dependent case (self-similarity of the Dehling-Taqqu type limit process).

To give a deeper insight into the continuity properties of the process $W$, we need a generalization of the Kolmogorov-Chentsov theorem. There are several multidimensional versions of this theorem in the literature, see e.g.  Mittmann and Steinwart \cite{mitt} and the references therein, but they deal with uniform continuity, while our theorem allows for H\"older continuity with different exponents in different directions. The proof is nevertheless completely analogous and is hence omitted.

\begin{prop}\label{prop1} Let $(Z_{t})_{t\in[0,1]^d}$ be a stochastic process such that for some $m\geq 1$, $c_1,\ldots,c_d,\beta_1,\ldots,\beta_d$ and for all $t=(t_1,\ldots,t_d)$, $s=(s_1,\ldots,s_d)$ we have
\begin{equation}
E\left[\left|Z_t-Z_s\right|^m\right]\leq\sum_{i=1}^dc_i\left|t_i-s_i\right|^{d+\beta_i}.
\end{equation}
Then for all $\gamma_1,\ldots,\gamma_d$ with $\gamma_i<\frac{\beta_i}{m}$, there exists a modification $\tilde{Z}$ of $Z$ and an almost surely finite random variable $C_{\gamma_1,\ldots,\gamma_d}$, such that for all $t=(t_1,\ldots,t_d)$, $s=(s_1,\ldots,s_d)$
\begin{equation}
\left|\tilde{Z}_t-\tilde{Z}_s\right|\leq C_{\gamma_1,\ldots,\gamma_d}\sum_{i=1}^d\left|t_i-s_i\right|^{\gamma_i}.
\end{equation}
\end{prop}
While for fixed $s$, the process $(W(s,t))_{t\in[0,1]}$ has the same modulus of continuity, no matter what the properties of the random walk $S_n$ are, it will turn out for higher $\alpha$, the limit process $(W(s,t))_{s\in[0,1]}$ for fixed $t$ is H\"older continuous with a higher exponent $\gamma$.
\begin{prop}\label{prop2} For any $\gamma<1-\frac{1}{2\alpha}$, $\gamma'<\frac{1}{2}$, there is a modification $\tilde{W}$ of $W$ and an almost surely finite random variable $C_{\gamma,\gamma'}$, such that for all $s,t,s',t'\in[0,1]$
\begin{equation}
\left|\tilde{W}(s,t)-\tilde{W}(s',t')\right|\leq C_{\gamma,\gamma'}\left(|s-s'|^\gamma+|t-t'|^{\gamma'}\right).
\end{equation}
\end{prop}
The exponent of H\"older continuity is linked to the exponent of self-similarity $b=1-\frac{1}{2\alpha}$. The same effect is known from fractional Brownian motion (see e.g. the book of Nourdin, \cite{nour}, p. 8).

\section{A Lemma on occupation times}

The occupation time $N_n(x)$ is defined as the number of visits of the random walk $(S_i)_{i=1,\ldots,n}$ to $x$:
\begin{equation}
 N_n(x):=\sum_{i=1}^n\mathds{1}_{\{S_i=x\}}.
\end{equation}
The following Lemma gives a relation to the local time of the limiting process of the random walk, similar to Lemma 6 of Kesten and Spitzer \cite{kest}. In our proofs, $C$ denotes a generic constant which might have different values in different inequalities, but does not depend on $n$.

\begin{lem}\label{lem1} For any $k\in\N$, $s_1,\ldots,s_k\in[0,1]$, the random vector
\begin{equation}
 \left(n^{-2+\frac{1}{\alpha}}\sum_{x\in\Z}N_{[ns_i]}(x)N_{[ns_j]}(x)\right)_{i,j\in\{1,\ldots,k\}}
\end{equation}
converges as $n\rightarrow\infty$ in distribution to
\begin{equation}
 \left(\int L_{s_i}(x)L_{s_j}(x)dx\right)_{i,j\in\{1,\ldots,k\}}.
\end{equation}

\end{lem}

\begin{proof} By the Cram\'er-Wold theorem, it suffices to show that for any $\theta_{ij}\in\R$, $i,j=1,\ldots, k$, we have as $n\rightarrow\infty$ the weak convergence
\begin{equation}
  n^{-2+\frac{1}{\alpha}}\sum_{i,j=1}^k\theta_{ij}\sum_{x\in\Z}N_{[ns_i]}(x)N_{[ns_j]}(x)\Rightarrow \sum_{i,j=1}^k\theta_{ij}\int L_{s_i}(x)L_{s_j}(x)dx.
\end{equation}
In order to show this, we will split the sum on the left side into several parts. Let $\tau>0$, $M>0$, $a(l,n)=\tau ln^{1/\alpha}$ and define
\begin{align}
Q(l,n)&:=n^{-2}\sum_{i,j=1}^k\theta_{ij}\sum_{a(l,n)\leq x,y<a(l+1,n)}N_{[ns_i]}(x)N_{[ns_j]}(y)\\
V(\tau,M,n)&:=\tau^{-1}\sum_{l=-M}^M Q(l,n)\\
U(\tau,M,n)&:=n^{-2+\frac{1}{\alpha}}\sum_{|x|>M\tau n^{1/\alpha}}\sum_{i,j=1}^k\theta_{ij}N_{[ns_i]}(x)N_{[ns_j]}(x).
\end{align}
We now can decompose the sum into four parts:
\begin{multline}
  n^{-2+\frac{1}{\alpha}}\sum_{i,j=1}^k\theta_{ij}\sum_{x\in\Z}N_{[ns_i]}(x)N_{[ns_j]}(x)\\
=V(\tau,M,n)+U(\tau,M,n)\displaybreak[0]\\
+\sum_{|l|\leq M}n^{-2+\frac{1}{\alpha}}\bigg(\sum_{a(l,n)\leq x<a(l+1,n)}\sum_{i,j=1}^k\theta_{ij}N_{[ns_i]}(x)N_{[ns_j]}(x)-\frac{n^2Q(l,n)}{[a(l+1,n)-a(l,n)]}\bigg)\displaybreak[0]\\
+\sum_{|l|\leq M}\left(n^{\frac{1}{\alpha}}[a(l+1,n)-a(l,n)]^{-1}-\frac{1}{\tau}\right)Q(l,n)\\
=V(\tau,M,n)+U(\tau,M,n)+I(\tau,M,n)+I\!\!I(\tau,M,n).
\end{multline}
We will treat the four summands separately. First note by Lemma 6 of Kesten and Spitzer \cite{kest} and the continuous mapping theorem, we have for $n\rightarrow\infty$ the convergence in distribution
\begin{equation}\label{line6}
 V(\tau,M,n)\Rightarrow \tau^{-1}\sum_{i,j=1}^k\theta_{ij}\sum_{|l|\leq M}\int_{\tau l}^{\tau (l+1)}L_{s_i}(x)dx\int_{\tau l}^{\tau (l+1)}L_{s_j}(x)dx=:V(\tau, M).
\end{equation}
For the summand $I(\tau,M,n)$, we introduce the mean occupation time of an interval $[a(l,n),a(l+1,n)]$:
\begin{equation}
 \bar{N}_{s_i,l}:=\frac{1}{[a(l+1,n)-a(l,n)]}\sum_{a(l,n)\leq x<a(l+1,n)}N_{[ns_i]}(x).
\end{equation}
Now we can rewrite $I(\tau,M,n)$ and apply the triangle inequality.
\begin{multline}
 \left|I(\tau,M,n)\right|\\
=\sum_{|l|\leq M}\sum_{a(l,n)\leq x<a(l+1,n)}n^{-2+\frac{1}{\alpha}}\left(\sum_{i,j=1}^k\theta_{ij}\left(N_{[ns_i]}(x)N_{[ns_j]}(x)-\bar{N}_{s_i,l}\bar{N}_{s_j,l}\right)\right)\displaybreak[0]\\
\leq \sum_{|l|\leq M}\sum_{a(l,n)\leq x<a(l+1,n)}n^{-2+\frac{1}{\alpha}}\left(\sum_{i,j=1}^k\theta_{ij}\left|N_{[ns_i]}(x)-\bar{N}_{s_i,l}\right|N_{[ns_j]}(x)\right)\displaybreak[0]\\
+\sum_{|l|\leq M}\sum_{a(l,n)\leq x<a(l+1,n)}n^{-2+\frac{1}{\alpha}}\left(\sum_{i,j=1}^k\theta_{ij}\bar{N}_{s_i,l}\left|N_{[ns_j]}(x)-\bar{N}_{s_j,l}\right|\right)\displaybreak[0]\\
\leq \theta^\star\sum_{|l|\leq M}\sum_{a(l,n)\leq x<a(l+1,n)}n^{-2+\frac{1}{\alpha}}\sum_{i,j=1}^k\left|N_{[ns_i]}(x)-\bar{N}_{s_i,l}\right|N_{[ns_j]}(x)\displaybreak[0]\\
+\theta^\star\sum_{|l|\leq M}\sum_{a(l,n)\leq x<a(l+1,n)}n^{-2+\frac{1}{\alpha}}\sum_{i,j=1}^k\bar{N}_{s_i,l}\left|N_{[ns_j]}(x)-\bar{N}_{s_j,l}\right|\displaybreak[0]\\
=:A_n+B_n
\end{multline}
with $\theta^\star:=\max\big\{|\theta_{i,j}|\big|\ 1\leq i,j\leq k\big\}$. By Lemma 1 of Kesten and Spitzer \cite{kest}, we have that
\begin{equation}
 E\left(N^2_{[ns_i]}(x)\right)\leq Cn^{2-\frac{2}{\alpha}},
\end{equation}
and by Lemma 3 of \cite{kest} in combination with their formula (2.26)
\begin{equation}
  E\left(N_{[ns_i]}(x)-N_{[ns_i]}(y)\right)^2\leq Cn^{1-\frac{1}{\alpha}}|x-y|^{\alpha-1}.
\end{equation}
Keep in mind that $a(l+1,n)-a(l.n)\leq C\tau n^{\frac{1}{\alpha}}$. Let $\|\cdot\|_2:=\sqrt{E[(\cdot)^2]}$ denote the $L_2$-norm. By the Cauchy-Schwarz inequality and the definition of $\bar{N}_{s_i,l}$, we obtain
\begin{multline}
 E\left|A_n\right|\leq \theta^\star\sum_{|l|\leq M}\sum_{a(l,n)\leq x<a(l+1,n)}n^{-2+\frac{1}{\alpha}}\sum_{i,j=1}^k\left\|N_{[ns_i]}(x)-\bar{N}_{s_i,l}\right\|_2\left\|N_{[ns_j]}(x)\right\|_2\displaybreak[0]\\
\leq  \theta^\star\sum_{|l|\leq M}\sum_{i,j=1}^k\sum_{a(l,n)\leq x,y<a(l+1,n)}\frac{n^{-2+\frac{1}{\alpha}}\left\|N_{[ns_j]}(x)\right\|_2}{[a(l+1,n)-a(l,n)]}\left\|N_{[ns_i]}(x)-N_{[ns_i]}(y)\right\|_2\displaybreak[0]\\
\leq C\theta^\star(2M+1)k^2\sum_{a(l,n)\leq x,y<a(l+1,n)}\frac{n^{-2+\frac{1}{\alpha}}\sqrt{n^{2-\frac{2}{\alpha}}}}{\tau n^{\frac{1}{\alpha}}}\sqrt{n^{1-\frac{1}{\alpha}}n^{\frac{1}{\alpha}(\alpha-1)}\tau^{\alpha-1}}\displaybreak[0]\\
=CM\sum_{a(l,n)\leq x,y<a(l+1,n)}n^{-\frac{2}{\alpha}}\tau^{-\frac{1}{2}-\frac{1}{2\alpha}}\leq CM\tau^{\frac{3}{2}-\frac{1}{2\alpha}}.
\end{multline}
With the same arguments and using the fact that
\begin{equation}
 \left\|\bar{N}_{s_i,l}\right\|_2\leq\frac{1}{[a(l+1,n)-a(l,n)]}\sum_{a(l,n)\leq x<a(l+1,n)}\left\|N_{[ns_i]}(x)\right\|_2,
\end{equation}
it follows that $E\left|B_n\right|\leq CM\tau^{\frac{3}{2}-\frac{1}{2\alpha}}$ and
\begin{equation}\label{line7}
E\left[I(\tau,M,n)\right]\leq CM\tau^{\frac{3}{2}-\frac{1}{2\alpha}}.
\end{equation}
For the next summand $I\!\!I(\tau,M,n)$, note that $Q(l,n)$ converges in distribution to $\sum_{i,j=1}^k\int_{\tau l}^{\tau(l+1)}L_{s_i}(x)dx\int_{\tau l}^{\tau(l+1)}L_{s_j}(x)dx$. Furthermore, $n^{\frac{1}{\alpha}}[a(l+1,n)-a(l,n)]^{-1}-1/\tau\rightarrow 0$ as $n\rightarrow\infty$ and consequently
\begin{equation}\label{line8}
 I\!\!I(\tau,M,n)=\sum_{|l|\leq M}\left(n^{\frac{1}{\alpha}}[a(l+1,n)-a(l,n)]^{-1}-\frac{1}{\tau}\right)Q(l,n)\xrightarrow{n\rightarrow\infty}0
\end{equation}
in probability. For the last summand, we have
\begin{equation}\label{line9}
 P\left(U(\tau,M,n)\neq 0\right)\leq P\left(N_n(x)>0 \ \text{for an} \ x \ \text{with}  |x|>M\tau n^{\frac{1}{\alpha}}\right)\leq \epsilon(M\tau),
\end{equation}
where $\epsilon(z)\rightarrow0$ as $z\rightarrow\infty$, see Lemma 1 of Kesten and Spitzer \cite{kest}. Note that the local time $L$ has almost surely a compact support, since the paths of the process $(S_s^\star)_{s\in[0,1]}$ are almost surely bounded, so we have for $V(\tau, M)$ defined in (\ref{line6}) the following limit
\begin{equation}\label{line10}
 V(\tau):=\lim_{M\rightarrow\infty}V(\tau, M)=\tau^{-1}\sum_{i,j=1}^k\theta_{ij}\sum_{l\in\Z}\int_{\tau l}^{\tau (l+1)}L_{s_i}(x)dx\int_{\tau l}^{\tau (l+1)}L_{s_j}(x)dx
\end{equation}
almost surely. By the almost sure continuity of the local time $L$ additionally
\begin{multline}\label{line11}
 \lim_{\tau\rightarrow0}V(\tau)=\lim_{\tau\rightarrow0}\sum_{i,j=1}^k\theta_{ij}\sum_{l\in\Z}\int_{\tau l}^{\tau (l+1)}L_{s_i}(x)\left(\tau^{-1}\int_{\tau l}^{\tau (l+1)}L_{s_j}(y)dy\right)dx\\
=\sum_{i,j=1}^k\theta_{ij}\int L_{s_i}(x)L_{s_j}(x)dx=:V.
\end{multline}
Finally, we combine the convergence of the different parts. Let $d(X,Y)$ denote the Prokhorov distance of the distributions of $X$ and $Y$ (so convergence with respect to $d$ is equivalent to weak convergence and $P\left(|X-Y|\geq\epsilon\right)\leq \epsilon$ implies $d(X,Y)\leq\epsilon$). For any $\epsilon>0$, choose $M,\tau>0$ in a way such that $M\tau$ is big enough and $\tau$, $M\tau^{\frac{3}{2}-\frac{1}{2\alpha}}$ are small enough to guarantee the following: $P\left(|V(\tau)-V|\geq\epsilon/6\right)\leq \epsilon/6$ by formula (\ref{line11}), $P\left(|V(\tau,M)-V(\tau)|\geq\epsilon/6\right)\leq \epsilon/6$ by formula (\ref{line10}) and $E\left[I(\tau,M,n)\right]\leq \frac{\epsilon^2}{36}$ by formula (\ref{line7}). Now we can choose $n_0\in\N$ with the help of (\ref{line6}) and (\ref{line8}), such that for all $n\geq n_0$ we have $d\left(V(\tau,M,n),V(\tau,M)\right)\leq \epsilon/6$ and $P\left(|I\!\!I(\tau,M,n)|\geq\epsilon/6\right)\leq \epsilon/6$ and arrive with the help of the triangle inequality at
\begin{multline}
d(V(\tau,M,n)+U(\tau,M,n)+I(\tau,M,n)+I\!\!I(\tau,M,n),V)\displaybreak[0]\\
\leq d(n^{-2+\frac{1}{\alpha}}\sum_{i,j=1}^n\theta_{ij}\sum_{x\in\Z}N_{[ns_i]}(x)N_{[ns_j]}(x),V(\tau,M,n)+U(\tau,M,n)+I(\tau,M,n))\displaybreak[0]\\
+d(V(\tau,M,n)+U(\tau,M,n)+I(\tau,M,n),V(\tau,M,n)+U(\tau,M,n))\\
+d(V(\tau,M,n)+U(\tau,M,n),V(\tau,M,n))+d(V(\tau,M,n),V(\tau,M))\\
+d(V(\tau,M),V(\tau))+d(V(\tau),V)\leq\epsilon.
\end{multline}

\end{proof}

\section{Proof of the Main Results}

\begin{proof}[Proof of Theorem \ref{theo1}] We will first prove the convergence of the finite dimensional distributions, tightness will be established later. We will make use of the Cram\'er-Wold theorem and show that for $\theta_1,\ldots,\theta_k\in\R$, $s_1,\ldots,s_k\in[0,1]$, $t_1,\ldots,t_k\in[0,1]$, we have as $n\rightarrow\infty$ the weak convergence
\begin{multline}
n^{-1+\frac{1}{2\alpha}}\sum_{j=1}^k\theta_j\sum_{i=1}^{[ns_j]}\left(\mathds{1}_{\{Y_i\leq t_j\}}-t_j\right)=n^{-1+\frac{1}{2\alpha}}\sum_{j=1}^k\theta_j\sum_{x\in\Z}N_{[ns_j]}(x)\zeta_j(x)\\
\Rightarrow \sum_{j=1}^k\theta_j\int L_{s_j}(x)dK(x,t_j),
\end{multline}
with $\zeta_j(x)=\mathds{1}_{\{\xi_x\leq t_j\}}-t_j$. For this, we will study the characteristic function and apply L\'evy's continuity theorem:
\begin{multline}
\varphi_n(\lambda):=E\left(\exp\bigg(i\lambda n^{-1+\frac{1}{2\alpha}}\sum_{j=1}^k\theta_j\sum_{x\in\Z}N_{[ns_j]}(x)\zeta_j(x)\bigg)\right)\\
=E\left(\prod_{x\in\Z}\exp\bigg(i\lambda n^{-1+\frac{1}{2\alpha}}\sum_{j=1}^k\theta_jN_{[ns_j]}(x)\zeta_j(x)\bigg)\right)\displaybreak[0]\\
=E\left(E\left(\prod_{x\in\Z}\exp\bigg(i\lambda n^{-1+\frac{1}{2\alpha}}\sum_{j=1}^k\theta_jN_{[ns_j]}(x)\zeta_j(x)\bigg)\Bigg|(X_n)_{n\in\N}\right)\right)\\
=E\left(\prod_{x\in\Z}E\left(\exp\bigg(i\lambda n^{-1+\frac{1}{2\alpha}}\sum_{j=1}^k\theta_jN_{[ns_j]}(x)\zeta_j(x)\bigg)\Bigg|(X_n)_{n\in\N}\right)\right),
\end{multline}
where we used the fact that the random variables $(\xi_x)_{x\in\Z}$ and thus also the random vectors $\big((\zeta_1(x),\ldots,\zeta_k(x))\big)_{x\in\Z}$ are independent and that inside conditional expectation, $(X_n)_{n\in\N}$ and thus $N_{[ns_1]}(x),\ldots,N_{[ns_k]}(x)$ are fixed. With $\varphi_{\zeta_1(0),\ldots,\zeta_k(0)}$, we denote the characteristic function of the random vector $(\zeta_1(0),\ldots,\zeta_k(0))$, so that
\begin{equation}
\varphi_n(\lambda)=E\left(\prod_{x\in\Z}\varphi_{\zeta_1(0),\ldots,\zeta_k(0)}\big(\lambda n^{-1+\frac{1}{2\alpha}}N_{[ns_1]}(x),\ldots, \lambda n^{-1+\frac{1}{2\alpha}}N_{[ns_k]}(x)\big)\right).
\end{equation}
The next step will be a Taylor expansion, so we have to gather some statements about the conditional moments. Keep in mind that $E\zeta_j(x)=0$ and thus
\begin{equation}
E\Big(\sum_{j=1}^k n^{-1+\frac{1}{2\alpha}}\theta_jN_{[ns_j]}(x)\zeta_j(x) \big|(X_n)_{n\in\N}\Big)=0.
\end{equation}
Furthermore
\begin{multline}
E\left(\bigg( n^{-1+\frac{1}{2\alpha}}\sum_{j=1}^k\theta_jN_{[ns_j]}(x)\zeta_j(x)\bigg)^2 \bigg|(X_n)_{n\in\N}\right)\displaybreak[0]\\
=\sum_{j,l=1}^kn^{-2+\frac{1}{\alpha}}\theta_j\theta_lN_{[ns_j]}(x)N_{[ns_l]}(x)\sigma_{jl}
\end{multline}
with $\sigma_{jl}:=\operatorname{Cov}(\zeta_j(x),\zeta_l(x))$. Finally, by Lemma 4 of Kesten and Spitzer \cite{kest}
\begin{equation}
\sup_{x\in\Z,s\in[0,1]}n^{-1+\frac{1}{2\alpha}}N_{[ns]}(x)\xrightarrow{n\rightarrow\infty}0
\end{equation}
in probability and by their Lemma 1 resepectively Lemma 2.1 of Guillotin-Plantard and Ladret \cite{guil}
\begin{align}
E\left(\sum_{x\in\Z}N_{[ns_j]}^2(x)\right)&\leq Cn^{2-\frac{1}{\alpha}},\\
E\left(\sum_{x\in\Z}N_n^3(x)\right)&\leq Cn^{3-\frac{2}{\alpha}}.
\end{align}
So we can conclude that
\begin{multline}\label{line5}
 \varphi_n(\lambda)\\
=E\Bigg(\prod_{x\in\Z}\bigg(1-\frac{\lambda^2}{2}\sum_{j,l=1}^kn^{-2+\frac{1}{\alpha}}\theta_j\theta_lN_{[ns_j]}(x)N_{[ns_l]}(x)\sigma_{jl}+O\big(n^{-3+\frac{3}{2\alpha}}N
 _n^3(x)\big)\bigg)\Bigg)\displaybreak[0]\\
\shoveleft=E\Bigg(\exp\bigg(\sum_{x\in\Z}\Big(-\frac{\lambda^2}{2}\sum_{j,l=1}^kn^{-2+\frac{1}{\alpha}}\theta_j\theta_lN_{[ns_j]}(x)N_{[ns_l]}(x)\sigma_{jl}\\
\shoveright {+o\big(n^{-2+\frac{1}{\alpha}}N_n^2(x)\big)+O\big(n^{-3+\frac{3}{2\alpha}}N _n^3(x)\big)\Big)\bigg)\Bigg)}\displaybreak[0]\\
\xrightarrow{n\rightarrow\infty}E\bigg(\exp\Big(-\frac{\lambda^2}{2}\sum_{j,l=1}^k\theta_j\theta_l\sigma_{jl}\int L_{s_j}(x)L_{s_l}(x)dx\Big)\bigg),
\end{multline}
where we used Lemma \ref{lem1} and the boundedness and continuity of the function $z\mapsto \exp(-z^2/2)$ to conclude that the expectation converges.

On the other hand, conditional on the L\'evy-process $S^\star$, the linear combination $\sum_{j=1}^k\theta_j\int L_{s_j}(x)dK(x,t_j)$ is Gaussian with variance
\begin{equation}
\sum_{j,l=1}^k\theta_j\theta_l\int L_{s_j}(x) L_{s_l}(x)\sigma_{jl}dx,
\end{equation}
as $\sigma_{jl}=\operatorname{Cov}(\zeta_j(x),\zeta_l(x))=\operatorname{Cov}(K(1,t_j),K(1,t_l))$ and the process $K$ is centered. This implies that
\begin{multline}
 E\left(\exp\Big(i\lambda \sum_{j=1}^k\theta_j\int L_{s_j}(x)dK(x,t_j)\Big)\right)\displaybreak[0]\\
= E\left(E\left(\exp\Big(i\lambda \sum_{j=1}^k\theta_j\int L_{s_j}(x)dK(x,t_j)\Big)\Big|S^\star\right)\right)\\
=E\left(\exp\Big(-\frac{1}{2}\lambda^2\sum_{j,l=1}^k\theta_j\theta_l\int L_{s_j}(x) L_{s_l}(x)\sigma_{jl}dx\Big)\right),
\end{multline}
and by (\ref{line5}) and L\'evy's continuity theorem the finite dimensional convergence follows. In order to prove tightness, we will establish a moment bound. First note that for all $j\in\Z$
\begin{align}
E\left(\mathds{1}_{\{\xi_j\leq t_1\}}-t_1-\mathds{1}_{\{\xi_j\leq t_2\}}+t_2\right)^2&\leq |t_1-t_2|,\\
E\left(\mathds{1}_{\{\xi_j\leq t_1\}}-t_1-\mathds{1}_{\{\xi_j\leq t_2\}}+t_2\right)^4&\leq |t_1-t_2|.
\end{align}
By Lemma 2.1 of Guillotin-Plantard and Ladret \cite{guil}, we have that
\begin{align}
E\left(\sum_{x\in\Z}N_n^2(x)\right)^2&\leq Cn^{4-\frac{2}{\alpha}},\\
E\left(\sum_{x\in\Z}N_n^4(x)\right)&\leq Cn^{4-\frac{3}{\alpha}}.
\end{align}
Now we obtain the following moment bound for all $n_1\leq n_2\leq n$ and $t_1,t_2\in[0,1]$ with $|t_1-t_2|\geq n^{-\frac{1}{\alpha}}$:
\begin{multline}\label{line4}
 E\left(\sum_{i=n_1+1}^{n_2}(\mathds{1}_{\{\xi_i\leq t_1\}}-t_1)-\sum_{i=n_1+1}^{n_2}(\mathds{1}_{\{\xi_i\leq t_2\}}-t_2)\right)^4\\
=E\left(E\bigg(\Big(\sum_{i=n_1+1}^{n_2}(\mathds{1}_{\{\xi_i\leq t_1\}}-t_1)-\sum_{i=n_1+1}^{n_2}(\mathds{1}_{\{\xi_i\leq t_2\}}-t_2)\Big)^4\bigg|(X_n)_{n\in\N}\bigg)\right)\\
\leq E\Bigg(\sum_{x\in\Z}N_{n_2-n_1}^4(x)|t_1-t_2|+\sum_{x\in\Z}\sum_{y\in\Z}N_{n_2-n_1}^2(x)N_{n_2-n_1}^2(y)|t_1-t_2|^2\Bigg)\\
\leq C\left((n_2-n_1)^{4-\frac{3}{\alpha}}|t_1-t_2|+C(n_2-n_1)^{4-\frac{2}{\alpha}}|t_1-t_2|^2\right)\\\leq C(n_2-n_1)^{4-\frac{2}{\alpha}}|t_1-t_2|^2.
\end{multline}
If $|t_1-t_2|\leq 2n^{-\frac{1}{\alpha}}$, we have by monotonicity that for any $t\in(t_1,t_2)$
\begin{multline}\label{line2}
 \left|\sum_{i=n_1+1}^{n_2}(\mathds{1}_{\{\xi_i\leq t\}}-t)-\sum_{i=n_1+1}^{n_2}(\mathds{1}_{\{\xi_i\leq t_1\}}-t_1)\right|\\
\leq \left|\sum_{i=n_1+1}^{n_2}\mathds{1}_{\{\xi_i\leq t\}}-\sum_{i=n_1+1}^{n_2}\mathds{1}_{\{\xi_i\leq t_1\}}\right|+(n_2-n_1)|t-t_1|\displaybreak[0]\\
\leq \left|\sum_{i=n_1+1}^{n_2}\mathds{1}_{\{\xi_i\leq t_2\}}-\sum_{i=n_1+1}^{n_2}\mathds{1}_{\{\xi_i\leq t_1\}}\right|+(n_2-n_1)|t_2-t_1|\displaybreak[0]\\
\leq \left|\sum_{i=n_1+1}^{n_2}(\mathds{1}_{\{\xi_i\leq t_2\}}-t_2)-\sum_{i=n_1+1}^{n_2}(\mathds{1}_{\{\xi_i\leq t_1\}}-t_1)\right|+2(n_2-n_1)|t_2-t_1|\\
\leq \left|\sum_{i=n_1+1}^{n_2}(\mathds{1}_{\{\xi_i\leq t_2\}}-t_2)-\sum_{i=n_1+1}^{n_2}(\mathds{1}_{\{\xi_i\leq t_1\}}-t_1)\right|+4n^{1-\frac{1}{\alpha}}.
\end{multline}
Following Bickel and Wichura \cite{bick}, we introduce for a two-parameter stochastic process $(V(s,t))_{s,t\in[0,1]}$ the notation
\begin{multline}\label{line3}
 w''_{\delta}(V)=\max\Big\{\sup_{\substack{0\leq t_1\leq t\leq t_2\leq 1\\t_2-t_1\leq \delta}}\min\left\{\|V(\cdot,t_2)-V(\cdot,t)\|_\infty,\|V(\cdot,t)-V(\cdot,t_1)\|_\infty\right\},\\
\sup_{\substack{0\leq s_1\leq s\leq s_2\leq 1\\s_2-s_1\leq\delta}}\min\left\{\|V(s_2,\cdot)-V(s,\cdot)\|_\infty,\|V(s,\cdot)-V(s_1,\cdot)\|_\infty\right\}\Big\},
\end{multline}
where $\|\cdot\|_\infty$ denotes the supremum norm. Now define the index set $D_n:=\left\{0,\frac{1}{n},\frac{2}{n}\ldots,1\right\}\times\left\{0,[n^{\frac{1}{\alpha}}]^{-1},2[n^{\frac{1}{\alpha}}]^{-1},\ldots,1\right\}$ and note that we have by (\ref{line2})
\begin{equation}\label{line1}
w''_{\delta}(n^{-1+\frac{1}{2\alpha}}W_n)\leq w''_{\delta}(n^{-1+\frac{1}{2\alpha}}W_{n|D_n})+4n^{-1+\frac{1}{2\alpha}}n^{1-\frac{1}{\alpha}},
\end{equation}
where $ w''_{\delta}(n^{-1+\frac{1}{2\alpha}}W_{n|D_n})$ is calculated by restricting all suprema in (\ref{line3}) to the set $D_n$. Now by Theorem 3 (and the remarks following their theorem) of Bickel and Wichura \cite{bick} together with (\ref{line4}), we can conclude that for any $\epsilon>0$
\begin{equation}
 P\left(\limsup_{n\rightarrow\infty}w''_{\delta}(n^{-1+\frac{1}{2\alpha}}W_{n|D_n})>\epsilon\right)\xrightarrow{\delta\rightarrow 0}0.
\end{equation}
It follows by (\ref{line1}), that
\begin{equation}
 P\left(\lim_{n\rightarrow\infty}w''_{\delta}(n^{-1+\frac{1}{2\alpha}}W_{n})>\epsilon\right)\xrightarrow{\delta\rightarrow 0}0.
\end{equation}
and thus the process is tight by Corollary 1 of \cite{bick}.
\end{proof}

\begin{proof}[Proof of Proposition \ref{prop2}] We will use Proposition \ref{prop1}, so we have to establish a moment inequality. Let be $m\in\N$ even, $s,s',t\in[0,1]$ with $s\leq s'$. Note that conditional on the process $S^{\star}$, the process $W$ is given by an It{\={o}} integral, and so it is Gaussian and we can apply the It{\={o}} isometry. Furthermore, note that the difference of local times $L_{s'}-L_s$ has the same distribution as $L_{s'-s}$ shifted by $S^\star(s)$. We obtain
\begin{multline}\label{line12}
E\left(W(s',t)-W(s,t)\right)^m=E\Big(\int(L_{s'}(x)-L_s(x))dK(x,t)\Big)^m\\
=E\left(E\bigg(\Big(\int (L_{s'}(x)-L_s(x))dK(x,t)\Big)^m\bigg|S^\star\bigg)\right)\displaybreak[0]\\
=E\left(E\bigg(\Big(\int L_{s'-s}(x)dK(x,t)\Big)^m\bigg|S^\star\bigg)\right)\\
=E\left(M(m,t)\Big(\int L^2_{s'-s}(x)dx\Big)^{\frac{m}{2}}\right),
\end{multline}
where $M(m,t)$ is the $m$-th moment of $K(1,t)$ and thus $M(m,t)\leq M_m$ for the $m$-th moment $M_m$ of a standard normal random variable. Now we gather some facts about local time. Obviously
\begin{equation}
 \int L_s(x)dx=s.
\end{equation}
By Theorem 1 of Davis \cite{davi}, we have that for $L^\star_s:=\sup_{x\in\R}L_s(x)$
\begin{equation}
 EL_s^{\star p}\leq C_ps^{p\frac{\alpha-1}{\alpha}}
\end{equation}
for a constant $C_p$ (we use the form of the inequality as stated by Lacey \cite{lace}, as there seems to be a misprint in \cite{davi}). Now we can proceed with the right side of (\ref{line12}):
\begin{multline}\label{line13}
 E\left(M(m,t)\Big(\int L^2_{s'-s}(x)dx\Big)^{\frac{m}{2}}\right)\leq M_mE\left(\left(\int L_{s'-s}(x)dx\right)^{\frac{m}{2}}\left(L^\star_{s'-s}\right)^{\frac{m}{2}}\right)\\
=M_m(s'-s)^{\frac{m}{2}}E\left(L^\star_{s'-s}\right)^{\frac{m}{2}}\leq M_mC_{m/2}(s'-s)^{\frac{m}{2}}(s'-s)^{\frac{m}{2}\frac{\alpha-1}{\alpha}}\\
\leq M_mC_{m/2}(s'-s)^{m(1-\frac{1}{2\alpha})}.
\end{multline}
Now let be $t,t'\in[0,1]$ with $t\leq t'$. Note that the process $(K(x,t')-K(x,t))_{t\in\R}$ is a (two-sided) Brownian motion with variance $\var(K(1,t')-K(1,t))\leq t'-t$. Now we proceed as above by conditioning on $S^\star$ and applying the It{\={o}}-isometry:
\begin{multline}\label{line14}
E\left(W(s,t')-W(s,t)\right)^m=E\left(\int L_s(x)d(K(x,t')-K(x,t))\right)^m\displaybreak[0]\\
=E\left(E\bigg(\Big(\int L_s(x)d(K(x,t')-K(x,t))\Big)^m\bigg|S^\star\bigg)\right)\displaybreak[0]\\
\leq E\left(M_m\left((t'-t)\int L_s^2(x)dx\right)^{\frac{m}{2}}\right)\displaybreak[0]\\
\leq M_m(t'-t)^{\frac{m}{2}}E\left(L^{\star}(s)\int L_s(x)dx\right)^{\frac{m}{2}}\displaybreak[0]\\
\leq M_m(t'-t)^{\frac{m}{2}}EL^{\star\frac{m}{2}}(1)\leq M(m,t)C_{m/2}(t'-t)^{\frac{m}{2}}.
\end{multline}
Combining (\ref{line13}) and (\ref{line14}), we arrive at
\begin{multline}
E\left(W(s',t')-W(s,t)\right)^m\displaybreak[0]\\
\leq 2^{m-1}\left(E\left(W(s',t')-W(s,t')\right)^m+E\left(W(s,t')-W(s,t)\right)^m\right)\displaybreak[0]\\
\leq 2^{m-1}M_mC_{m/2}(s'-s)^{m(1-\frac{1}{2\alpha})}+2^{m-1}M_mC_{m/2}(t'-t)^{\frac{m}{2}}.
\end{multline}
Now for any $\gamma<1-\frac{1}{2\alpha}$, $\gamma'<\frac{1}{2}$, we can choose $m$ large enough such that
\begin{equation}
 \gamma<\frac{{m(1-\frac{1}{2\alpha})}-2}{m}, \ \ \ \gamma'<\frac{\frac{m}{2}-2}{m},
\end{equation}
and the statement of this proposition follows from Proposition \ref{prop1}.

\end{proof}

\section*{Acknowledgement}
The research was supported by the DFG Sonderforschungsbereich 823 (Collaborative Research Center) {\em Statistik nichtlinearer dynamischer Prozesse}. I would like to thank Brice Franke and Norman Lambot for their careful reading of the article and their remarks. I am very grateful to the anonymous referee for his comments, which have helped to improve and clarify this manuscript.

\end{document}